\documentclass[10pt,psamsfonts]{amsart}
\usepackage{amsthm,amsmath,amsfonts,amssymb,amsthm,graphicx,epstopdf,color,pifont,subfigure,flafter,bm,amscd}
\usepackage{graphicx,overpic}
\usepackage[margin=3.6cm]{geometry}
\usepackage[colorlinks=true, linkcolor=blue, citecolor=red]{hyperref}
\usepackage{multicol}

\def\p{\partial}
\def\R{\mathbb R}
\def\sss{\mathbb S}

\newtheorem {theorem} {Theorem}
\newtheorem {proposition} [theorem]{Proposition}

\newtheorem {lemma}  [theorem]{Lemma}

\usepackage{float}
\usepackage{tikz, wrapfig}
\tikzset{node distance=3cm, auto}
\allowdisplaybreaks

\begin{document}

\title[The dynamics of the Ehrhard-M\"uller system]
{The dynamics of the Ehrhard-M\"uller system\\ with invariant algebraic surfaces}

\author[Jaume Llibre and Gabriel Rondón]{Jaume Llibre$^{1}$ and Gabriel Rondón$^{1}$}

\address{$^{1}$ Departament de Matemàtiques, Edifici Cc, Universitat Autònoma de Barcelona, 08193 Bellaterra, Barcelona, Catalonia, Spain}
\email{jaume.llibre@uab.cat}

\email{garv202020@gmail.com}



\begin{abstract}
In this paper we study the global dynamics of the Ehrhard-M\"uller differential system 
\[
\dot{x} = s(y - x), \quad \dot{y} = rx - xz - y + c, \quad \dot{z} = xy - z,
\]
where \(s\), \(r\) and \(c\) are real parameters, and \(x\), \(y\), and \(z\) are real variables. We classify the invariant algebraic surfaces of degree $2$ of this differential system. After we describe the phase portraits in the Poincaré ball of this differential system having one of this invariant algebraic surfaces. 

The Poincaré ball is the closed unit ball in $\R^3$ whose interior has been identified with $\R^3$, and his boundary, the $2$-dimensional sphere $\sss^2$, has been identified with the infinity of $\R^3$. Note that in the space $\R^3$ we can go to infinity in as many as directions as points has the sphere $\sss^2$. A polynomial differential system as the Ehrhard-M\"uller system  can be extended analytically to the Poincaré ball, in this way we can study its dynamics in a neigborhood of infinity. Providing these phase portraits in the Poincaré ball we are describing the dynamics of all orbits of the Ehrhard-M\"uller system having an invariant algebraic surface of degree $2$.
\end{abstract}

\maketitle

\section{Introduction and statement of the main results}

Consider the polynomial differential system
\begin{equation}\label{main_eq_1}
\dot{x}=s(y-x),\quad \dot{y}=rx-xz-y+c,\quad\dot{z}=xy-z,
\end{equation}
where  $r,s,c$ are real parameters, $(x,y,z)\in\mathbb{R}^3$ and the dot $\cdot$ denotes the derivative of the functions $x(t)$, $y(t)$ and $z(t)$ with respect to the real variable $t$, usually called the time. Furthermore, we will denote by
\begin{equation}\label{VF}
X= s(y-x)\dfrac{\p}{\p x}+ (rx-xz-y+c)\dfrac{\p}{\p y}+ (xy-z)\dfrac{\p}{\p z}
\end{equation}
the vector field associated to system \eqref{main_eq_1}. This differential system if $c = 0$ becomes the classical Lorenz system with $b=1$. In what follows we will focus our attention when $c$ is non-zero. 

The nonlinear dynamic properties of system \eqref{main_eq_1} have been extensively analyzed in various articles; see, for instance, \cite{EM, GWR1, GWR2, PLB, We, WGR}. Specifically, \cite{EM} established the stability and bifurcation behaviors of the Ehrhard-M\"uller system, consistent with both symmetric and non-symmetric heating experiments. Meanwhile,  \cite{PLB} investigated the periodicity of the system, nonlinear dynamics, and chaos, deriving a periodicity diagram. Additional studies \cite{GWR1, GWR2, We, WGR} explored the chaotic properties of system \eqref{main_eq_1} through numerical and experimental approaches.
However, the full range of dynamical behaviors exhibited by the Ehrhard--M\"uller system has not yet been fully characterized. This gap motivates our present investigation. 


Comparing the results on the Lorenz system with the ones of the Ehrhard-M\"uller system some natural questions arise: what are the invariant algebraic surfaces of degree $2$ of the Ehrhard-M\"uller system? What is the global dynamics in the Poincaré ball of the Ehrhard-M\"uller system having some of these invariant surfaces? The objective of this paper is to answer these two questions.

Let \( f=f(x,y,z) \) be a real polynomial. We recall that \(f(x,y,z)=0\) is an {\it invariant algebraic surface} of the differential system \eqref{main_eq_1} or of its vector field $X$ if 
\begin{equation}\label{IAC}
X f= kf,
\end{equation}
where \( k = k(x,y,z) \) is a real polynomial of degree at most 1, called the \textit{cofactor} of \( f \). This means that if an orbit of the differential system \eqref{main_eq_1} passes through a point on this surface, then the whole orbit is contained in the surface.

\begin{theorem}\label{prop_alg}
All the invariant algebraic surfaces $f(x, y, z)=0$ of degree $2$ of the Ehrhard-M\"uller system \eqref{main_eq_1} are given in the Table $1$.
\begin{table}[h]
\begin{center}
\begin{tabular}{| c ||c| c | c |}
\hline
Case & $s,r,c$ & $f(x,y,z)$ & $k$ \\
\hline\hline
$(a)$  & $s=1/2$ & $x^2-z$ & $-1$ \\
\hline
$(b)$  & $s=2,$ $r=0$  & $y^2 + z^2-cx$ & $-2$ \\
\hline
$(c)$ & $r=0,$ $c=0$  & $y^2 + z^2$ & $-2$ \\
\hline
$(d)$ & $s=1,$ $c=0$  & $y^2 + z^2-r x^2$ & $-2$ \\
\hline
\end{tabular}
\end{center}
\vspace{0.2cm} \caption{The invariant algebraic surfaces $f(x, y, z)=0$ of degree 2 of system \eqref{main_eq_1}.}\label{tb_2}
\end{table}
\end{theorem}

A function of the form $I=I(x,y,z)e^{st}$ is a {\it Darboux invariant} for the differential system \eqref{main_eq_1} if $dI/dt=0$ on the solutions of this differntial system. In other words, a  Darboux invariant is a kind of first integral depending on the time. 

\begin{proposition}\label{p1}
The Ehrhard-M\"uller system \eqref{main_eq_1} has the Darboux invariant
\item[(a)] $(x^2-z)e^t$ if $s=1/2$;

\item[(b)] $(y^2 + z^2-cx)e^{2t}$ if $s=2$ and $r=0$;

\item[(c)] $(y^2 + z^2)e^{2t}$ if $r=c=0$;

\item[(d)] $(y^2 + z^2-r x^2)e^{2t}$ if $s=1$ and $c=0$.
\end{proposition}

\begin{figure}[h]
\begin{overpic}[scale=0.4]{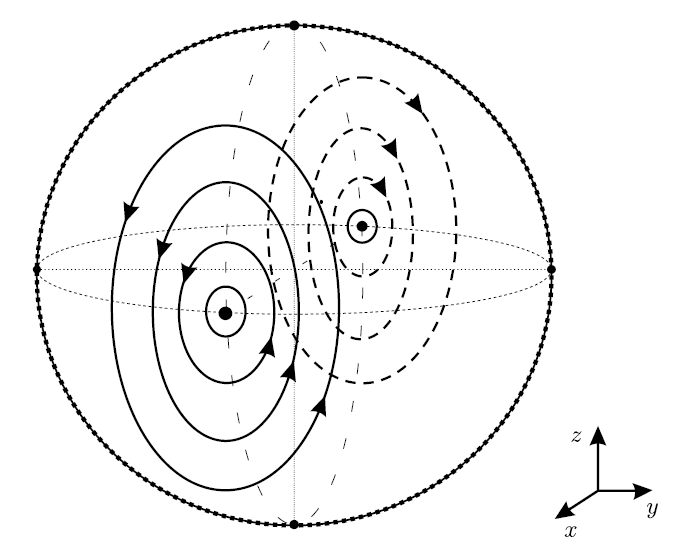}
        \put(28,29){\scriptsize{ $P$}}
		\end{overpic}
\caption{\footnotesize{ Phase portrait of the Ehrhard-M\"uller system \eqref{main_eq_1} at the infinity sphere.
}}
\label{infinity}
\end{figure}

For a definition of the Poincaré ball see Subsubsection \ref{s212}.

\begin{proposition}\label{p2}
For all values of the parameters $r,s,c\in \mathbb{R}$ the phase portrait of the Ehrhard-M\"uller system \eqref{main_eq_1} on the sphere $\sss^2$ of the infinity is given in Figure \ref{infinity}. On this figure there are two centers at the endpoints of the $x$--axis and a circle of equilibria at the endpoints of the plane $x=0$.
\end{proposition}

In the next theorem we describe the phase portrait in the Poincar\'e ball of the Ehrhard-M\"uller system \eqref{main_eq_1} having the invariant algebraic surface $x^2-z=0$. For the usual definitions $\omega$-limit set and $\alpha$-limit set of an orbit of the differential system \eqref{main_eq_1} see Subsection \ref{Darboux}.

\begin{theorem}\label{ta}
The following statements hold for the Ehrhard-M\"uller system \eqref{main_eq_1} having the invariant algebraic surface $x^2-z=0$.
\begin{itemize}
\item[(i)] The phase portraits on the invariant algebraic surface $x^2-z=0$ are topological equivalent to one of the phase portraits of Figure \ref{F2}. 

\item[(ii)] The $\alpha$-limit of an orbit contained in the interior of the Poincaré ball and in the region $x^2-z<0$ is the infinite equilibrium point at the origin of the local chart $U_3$, while the $\alpha$-limit of an orbit contained in the interior of the Poincaré ball and in the region $x^2-z>0$ is the infinite equilibrium point at the origin of the local chart $V_3$.

\item[(iii)] The $\omega$-limit of an orbit contained in the interior of the Poincaré ball and outside of the invariant surface $x^2-z=0$ is either one of the stable finite equilibrium points contained in the surface $x^2-z=0$, or one of the two infinite equilibrium points at the endpoints of the $y$-axis contained in the surface $x^2-z=0$. 
\end{itemize}
\end{theorem}
\begin{figure}[h]
\begin{overpic}[scale=0.3]{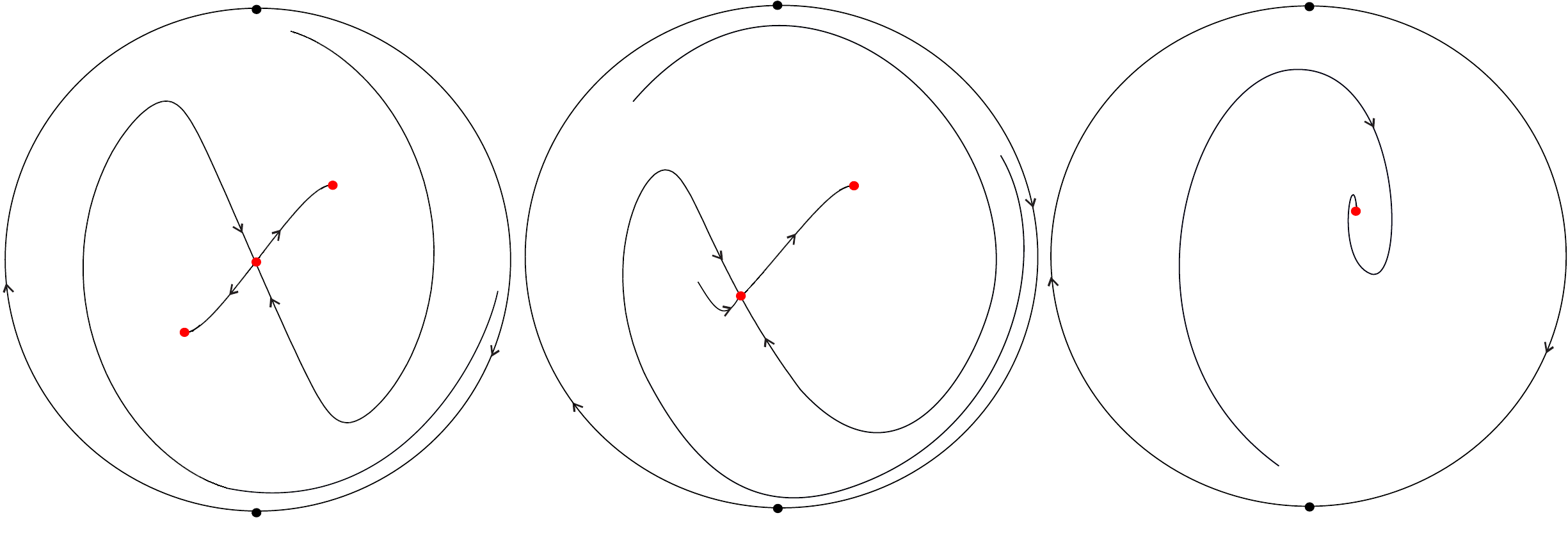}
        \put(14,0){$\Delta>0$}
        \put(49,0){$\Delta=0$}
        \put(84,0){$\Delta<0$}
		\end{overpic}
\caption{\footnotesize{ The topological phase portraits on the invariant algebraic surface $x^2-z=0$.
}}
\label{F2}
\end{figure}
In the following theorem we describe the phase portrait in the Poincar\'e ball of the Ehrhard-M\"uller system \eqref{main_eq_1} having the invariant algebraic surface $y^2 + z^2-cx=0$. 

\begin{theorem}\label{tb}
The following statements hold for the Ehrhard-M\"uller system \eqref{main_eq_1} having the invariant algebraic surface $y^2 + z^2-cx=0$.
\begin{itemize}
\item[(i)] The phase portraits on the invariant algebraic surface \( y^2 + z^2 - cx = 0 \) are topologically equivalent to the phase portrait in Figure \ref{fig_surface_2}. So all the orbits start at infinity and end in the stable focus $Q.$
	
\item[(ii)] The $\alpha$-limit of an orbit contained in the interior of the Poincaré ball and outside of $y^2 + z^2-cx=0$ is the infinity equilibrium point \( P \) at the origin of the chart \( U_1 \), and it is located at the endpoint of the positive \( x \)-axis.

\item[(iii)] The $\omega$-limit of an orbit contained in the interior of the Poincaré ball and outside of the invariant surface $y^2 + z^2-cx=0$ is the stable focus $Q$. 
\end{itemize}
\end{theorem}

\begin{figure}[h]
\begin{overpic}[scale=0.4]{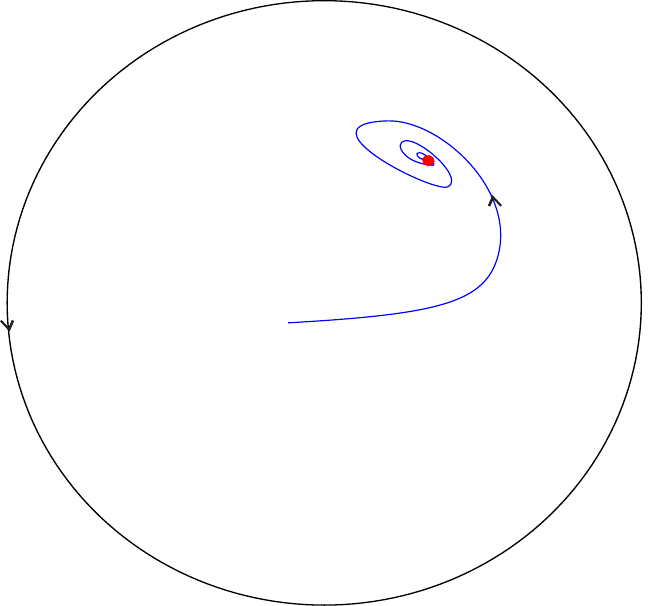}
        \put(55,62){\scriptsize{$Q$}}
		\end{overpic}
\caption{\footnotesize{
The topological phase portraits on the invariant algebraic surface $y^2+z^2-cx=0$. The $Q$ point is the stable focus of system \eqref{main_eq_1} on the algebraic surface $y^2+z^2-cx=0.$}}
\label{fig_surface_2}
\end{figure}

We note that the Ehrhard-M\"uller system \eqref{main_eq_1} having the invariant straight line $y^2+z^2=0$, i.e. the $x$-axis, coincides with the Lorenz system $\dot{x}=s(y-x)$, $\dot{y}=-xz-y$, $\dot{z}=xy-z$, while the Ehrhard-M\"uller system \eqref{main_eq_1} having the invariant surface $y^2+z^2-r x^2=0$, coincides with the Lorenz system $\dot{x} = y - x$, $\dot{y} = rx - xz - y$, $\dot{z} = xy - z$. Hence the phase portrait in the Poincar\'e ball of the Ehrhard-M\"uller system \eqref{main_eq_1} having these two invariant algebraic surfaces have been studied in the papers \cite{CZ,LMS}.

The paper is organized as follows. In Subsection \ref{s21} we introduce the Poincaré compactification in $\R^2$ and $\R^3$. Subsection  \ref{sec:blowup} is dedicated to the vertical blow-ups that are essential for analyzing the local dynamics of the planar polynomial differential systems near equilibrium points having its linear part identically zero. In Subsection \ref{sec:ben_crit} we examine the Bendixson and the Bendixson-Dulac criteria, that aids in analyzing the non-existence of periodic orbits in the planar differential systems. Subsection \ref{Darboux} is dedicated to introduce the Darboux invariants. The proofs of Theorem \ref{prop_alg} and Propositions \ref{p1} and \ref{p2} are given in Section \ref{s3}. In Sections \ref{sta} and \ref{stb} we prove Theorems \ref{ta} and \ref{tb}, respectively.

\section{Preliminaries}\label{sec:Preliminaries}

This section is devoted to establish some basic results that will be used throughout the paper. We divide it in four subsections.

\subsection{Poincar\'e Compactification}\label{poincare_comp}\label{s21}

For determining the global dynamics of a polynomial differential system in $\R^3$ or in an invariant algebraic surface we need to know the local phase portraits of its finite and infinite equilibrium points in the Poincaré ball and in the Poincaré disc, respectively.

\subsubsection{Poincar\'e Compactification in $\mathbb{R}^2$}\label{s211}

Let $\sss^2=\{\mathbf{z}\in\mathbb{R}^3:||\mathbf{y}||=1\}$ be the unit sphere of $\mathbb{R}^3$ centered at the origin of coordinates. From \cite{DLA}, we know that a polynomial vector field $X$ in $\R^2$ induces two copies of $X$ in the sphere $\sss^2$, one in the open northern hemisphere and the other in the open southern hemisphere of the sphere $\mathbb{S}^2$. This vector field on $\mathbb{S}^2\setminus \mathbb{S}^1$ can be extended analytically to a vector field $p(X)$ on the whole sphere $\mathbb{S}^2$. Here $\sss^1=\{\mathbf{y}\in \sss^2:y_3=0\}$ denotes the equator of $\mathbb{S}^2$, that can be identified with the infinity of $\R^2$. The vector field $p(X)$ allows to study the dynamics of the vector field $X$ in the neighbourhood of infinity, i.e. in the neighbourhood of the equator $\sss^1.$

To get the analytical expression for $p(X)$ we should consider the sphere $\sss^2$ as a smooth manifold. In this context, it is enough to choose the six local charts given by $U_i=\{\mathbf{y}\in \sss^2: y_i>0\}$ and $V_i=\{\mathbf{y}\in \sss^2: y_i<0\}$, for $i=1, 2, 3$, with the corresponding coordinate maps $\varphi_i: U_i\rightarrow \mathbb{R}^2$ and $\psi_i: V_i\rightarrow \mathbb{R}^2$, defined by $\varphi_k(\mathbf{y}) = \psi_k(\mathbf{y}) = (y_m/y_k, y_n/y_k)$ for $m < n$ and $m, n\ne k$. Denote by $(u, v)$ the local coordinates on $U_i$ and $V_i$ for $i=1,2,3.$ If $X=(P,Q)$ from \cite{DLA} the vector field $p(X)$ in these local charts is
\begin{equation}\label{poincare_comp}
\begin{array}{rl}		(\dot{u},\dot{v})=&\left(v^n\left(-uP\left(\dfrac{1}{v},\dfrac{u}{v}\right)+Q\left(\dfrac{1}{v},\dfrac{u}{v}\right)\right),-v^{n+1}P\left(\dfrac{1}{v},\dfrac{u}{v}\right)\right) \quad\text{in }\quad U_1;\vspace{0.3cm}\\
(\dot{u},\dot{v})=&\left(v^n\left(-uQ\left(\dfrac{u}{v},\dfrac{1}{v}\right)+P\left(\dfrac{u}{v},\dfrac{1}{v}\right)\right),-v^{n+1}Q\left(\dfrac{u}{v},\dfrac{1}{v}\right)\right) \quad\text{in }\quad U_2;\vspace{0.3cm}\\
(\dot{u},\dot{v})=&\left(P(u,v),Q(u,v)\right) \quad\text{in }\quad U_3,
\end{array}
\end{equation}
where $n$ is the degree of the polynomial vector field $X$. We emphasize that the expressions of the vector field $p(X)$ in the local chart $(V_i,\psi_i)$ is the same that as in the local card $(U_i,\varphi_i)$ multiplied by $(-1)^{n-1}$ for $i=1,2,3.$

The points of the infinity in all local charts are of the form $(u,0)$. The infinity $\mathbb{S}^1$ is invariant under the flow of $p\left(X\right)$.

\subsubsection{Poincar\'e Compactification in $\mathbb{R}^3$}\label{s212}

Let $\mathbb{S}^3=\{\mathbf{y}\in\mathbb{R}^4:||\mathbf{y}||=1\}$ be the unit sphere in $\mathbb{R}^4$ centered at the origin of coordinates. From \cite{CL}, we know that a polynomial vector field $X$ of degree $n$ in $\R^3$ induces two copies on the sphere $\mathbb{S}^3$, one in the open northern hemisphere and the other in the southern hemisphere of the sphere $\mathbb{S}^3$. This induced vector field on $\mathbb{S}^3\setminus \sss^2$ can be extended analytically to a vector field $p(X)$ on the whole sphere $\sss^3$. Here  $\mathbb{S}^2=\{\mathbf{y}\in \mathbb{S}^3:y_4=0\}$, the equator of $\sss^3$, is identified with the infinity of $\R^3$. The vector field $p(X)$ allows to study the dynamics of the vector field $X$ in the neighborhood of infinity, i.e., in the neighborhood of the equator $\mathbb{S}^2.$

To get the analytical expression for $p(X)$ we shall consider the sphere $\sss^3$ as a smooth manifold. It is enough to choose the 8 local charts  $U_i=\{\mathbf{y}\in \mathbb{S}^3: y_i>0\}$ and  $V_i=\{\mathbf{y}\in \mathbb{S}^3: y_i<0\}$, for $i=1, 2, 3, 4$,  with the corresponding coordinate maps $\varphi_i: U_i\rightarrow \mathbb{R}^3$ and $\psi_i: V_i\rightarrow \mathbb{R}^3$, defined by $\varphi_k(\mathbf{y}) = \psi_k(\mathbf{y}) = (y_\ell/y_k, y_m/y_k, y_n/y_k)$ for $\ell<m < n$ and $\ell,m, n\ne k$. Denote by $(z_1, z_2, z_3)$ the local coordinates on $U_i$ and $V_i$ for $i=1,2,3,4.$ If $X=(P_1,P_2,P_3)$ from \cite{Du} the vector field $p(X)$ in these local charts is
\begin{equation}\label{poincare_comp_3}
\begin{array}{rl}
z_3^{n}(-z_1P^1_1+P^2_1,-z_2P^1_1+P^3_1,-z_3P^1_1) \quad\text{in }\quad U_1;\vspace{0.2cm}\\
z_3^{n}(-z_1P^2_2+P^1_2,-z_2P^2_2+P^3_2,-z_3P^2_2) \quad\text{in }\quad U_2;\vspace{0.2cm}\\
z_3^{n}(-z_1P^3_3+P^1_3,-z_2P^3_3+P^2_3,-z_3P^3_3) \quad\text{in }\quad U_3;\vspace{0.2cm}\\
(P^1_4,P^2_4, P^3_4) \quad\text{in }\quad U_4,
\end{array}
\end{equation}
where $P^i_1=P^i\left(\dfrac{1}{z_3},\dfrac{z_1}{z_3},\dfrac{z_2}{z_3}\right),$ $P^i_2=P^i\left(\dfrac{z_1}{z_3},\dfrac{1}{z_3},\dfrac{z_2}{z_3}\right),$ $P^i_3=P^i\left(\dfrac{z_1}{z_3},\dfrac{z_2}{z_3},\dfrac{1}{z_3}\right),$ $P^i_4=P^i(z_1,z_2,z_3),$ $\Delta(z)=(1+\sum_{i=1}^3z_i^2)^{1/2}$ and $n$ is the degree of the polynomial vector field $X$.

The expression for $p(X)$ in the local chart $V_i$ is the same as in $U_i$ multiplied by $(-1)^{n-1}$ for $i=1,2,3,4$.

The points of the infinity in all local charts are of the form $(z_1,z_2,0)$. The infinity $\mathbb{S}^2$ is invariant under the flow of $p\left(X\right)$.

The proof of the following result can be found in \cite[Lemma 2.1]{LMS}.

\begin{lemma}\label{boundary}
Let $f(x_1, x_2, x_3)=0$ be an algebraic surface of degree $m$ in $\R^3$. The extension of this surface to the boundary of the Poincar\'e ball is the curve described by the intersection of the two surfaces
$$
y^m_4f\left(\dfrac{x_1}{y_4},\dfrac{x_2}{y_4},\dfrac{x_3}{y_4}\right) = 0, \quad y_4= 0.
$$
\end{lemma}

\subsection{Vertical Blow Up}\label{sec:blowup}

Consider a polynomial differential system in $\R^2$ described by the equations:
\begin{equation}\label{e1}
\dot{x} = P(x, y) = P_n(x, y) + \ldots,\quad
\dot{y} = Q(x, y) = Q_n(x, y) + \ldots,
\end{equation}
where \(P\) and \(Q\) are coprime polynomials, \(P_n\) and \(Q_n\) are homogeneous polynomials of degree \(n \in \mathbb{N}\), and the dots indicate higher-order terms in \(x\) and \(y\). Since \(n > 0\), the origin is an equilibrium point of the system \eqref{e1}. The \textit{characteristic directions} at the origin correspond to the straight lines through the origin, determined by the real linear factors of the homogeneous polynomial:
\begin{equation}\label{RR}
R_n(x,y) = P_n(x,y)y - Q_n(x,y)x.
\end{equation}
It is established that orbits starting or ending at the origin do so tangentially along these characteristic directions. For further details on characteristic directions, refer to \cite{ALGM}.

Assuming there is an equilibrium point at the origin, as described in system \eqref{e1}, and that this equilibrium is linearly zero, we will analyze its local phase portrait using vertical blow-ups.

We define a vertical blow-up in the \(y\) direction through the variable transformation \((u, v) = (x, y/x)\). This transformation maps the origin of system \eqref{e1} to the straight line \(u =x= 0\). By studying the dynamics of the differential system in a neighborhood of this line, we effectively analyze the local phase portrait of the equilibrium point at the origin of system \eqref{e1}. However, prior to performing a vertical blow-up, we must ensure that the direction \(x = 0\) is not a characteristic direction of the origin. If \(x = 0\) is indeed a characteristic direction, we perform an appropriate twist using the transformation \((x, y) = (v + \alpha u,v)\) with \(\alpha \neq 0\).

\subsection{Bendixson-Dulac criterion}\label{sec:ben_crit}

In what follows we recall a result that allows to establish the non-existence of closed orbits for a $C^1$ differential systems in the plane $\R^2$ of the form  \begin{equation}\label{Ben_C}
x'=P(x,y),\, y'=Q(x,y).
\end{equation}
The criterion is due to Bendixson \cite{Be} (see also Theorem 7.10 of \cite{DLA}) and says: Let $D$ be a simply-connected domain in $\R^2$. Assume that the divergence of the system
$$
\frac{\p P}{dx}+\frac{\p Q}{dy}
$$
has constant sign in $D$, except perhaps in a subset of $D$ of zero Lebesgue measure. Then the differential system \eqref{Ben_C} has no closed orbits contained in $D$.

This criterion was generalized by H. Dulac, see Theorem 7.12 of \cite{DLA} as follows: If $D$ is a simply-connected domain in $\R^2$, if there exists a $C^1$ function $f(x,y)$ such that
$$
\frac{\p (fP)}{dx}+\frac{\p (fQ)}{dy}
$$
has constant sign in $D$, except perhaps in a subset of $D$ of zero Lebesgue measure. Then the differential system \eqref{Ben_C} has no closed orbits contained in $D$.

\subsection{Darboux invariants}\label{Darboux}

Let $\varphi_p(t)$ be the solution of the differential system \eqref{main_eq_1} passing through the point $p \in \R^3$, defined on its maximal
interval $I_p=(\alpha_p,\omega_p)$. If $\omega_p=\infty$ we define the set 
\[
\omega(p)=\{q\in\Delta\,:\,\,\mbox{there exist }\{t_n\}\text{ with }t_n\to\infty \,\text{ and } \varphi(t_n)\to q\text{ when }n\to\infty\}. 
\]
In the same way, if $\alpha_p=-\infty$ we define the set
\[
\alpha(p)=\{q\in\Delta\,:\,\,\mbox{there exist }\{t_n\}\text{ with }t_n\to-\infty \,\text{ and } \varphi(t_n)\to q\text{ when }n\to\infty\}. 
\]

The sets $\omega(p)$ and $\alpha(p)$ are called the {\it $\omega$--limit set} (or simply {\it $\omega$--limit} and the {\it $\alpha$--limit set} (or {\it
$\alpha$--limit}) of $p$, respectively.

The existence of a Darboux invariant provides information about the $\omega$- and $\alpha$-limit sets of all orbits of system \eqref{main_eq_1}. More precisely, we have the following result.

\begin{proposition}\label{p3}
Let $I(x,y,z,t) = f(x,y,z)e^{st}$ be a Darboux invariant of system \eqref{main_eq_1}. Let $p \in \R^3$ and $\varphi_p(t)$ be the solution of system \eqref{main_eq_1} with maximal interval $(\alpha_p,\omega_p)$ such that $\varphi_p(0)=p$. Then
\begin{itemize}
\item[(i)] If $\omega_p = \infty$ then $\omega(p) \subset \{f(x,y,z) = 0\} \cup\sss^2$.

\item[(ii)] If $\alpha_p=-\infty$ then $\alpha(p) \subset \{f(x,y,z) = 0\} \cup\sss^2$.
\end{itemize}
Here $\sss^2$ denotes the sphere of the infinity of the Poincaré ball.
\end{proposition}

For a proof of Proposition \ref{p3} see Proposition 5 of \cite{LO}.

\section{Proof of Theorem $\ref{prop_alg}$ and Propositions $\ref{p1}$ and $\ref{p2}$}\label{s3}

In this section, we provide the formal proofs for some results previously stated. First, we demonstrate Theorem $\ref{prop_alg}$ that characterizes the invariant algebraic surfaces of degree 2 in the Ehrhard-Müller system.

\begin{proof}[Proof of Theorem $\ref{prop_alg}$]
Let $X$ be the vector field defined by the differential system \eqref{VF}. To find all invariant algebraic surfaces $f(x, y, z)=0$ of degree 2 of system \eqref{main_eq_1}, we must find the coefficients $a_i,$ $i=1,\dots,9$ and $k_j,$ $j=1,2,3,4$ that satisfy equation \eqref{IAC}
with
$$
f(x,y,z)=a_0+a_1x+a_2y+a_3z+a_4x^2+a_5xy+a_6xz+a_7y^2+a_8yz+a_9z^2
$$
and
$$
k(x,y,z)=k_1x+k_2y+k_3z+k_4.
$$
Substituting $f$ and $k$ into equation \eqref{IAC}, and equating both polynomials we obtain the following system of equations whose variables are the coefficients $a_i,$ $k_j$, $r,$ $s,$ and $c$:
$$
\begin{array}{ll}
(i)\,\, a_2 c - a_0 k_4=0;                                         & (ii)\,\,a_5 c - a_1 k_4 - a_0 k_1 + a_2 r - a_1 s=0; \\
(iii)\,\,a_2 - 2 a_7 c + a_2 k_4 + a_0 k_2 - a_1 s=0;              & (iv)\,\, a_3 - a_8 c + a_3 k_4 + a_0 k_3=0;\\
(v)\,\,a_4 k_4 + a_1 k_1 - a_5 r + 2 a_4 s=0;                      & (vi)\,\,2 a_7 + a_7 k_4 + a_2 k_2 - a_5 s=0;\\
(vii)\,\,2 a_9+ a_9 k_4+a_3 k_3=0;                                 & (viii)\,\,2 a_8 + a_8 k_4 + a_3 k_2 + a_2 k_3 - a_6 s=0; \\
(ix)\,\,a_2 + a_6 + a_6 k_4 + a_3 k_1 + a_1 k_3 - a_8 r + a_6 s=0; & (x)\,\, a_8 - a_7 k_1 - a_5 k_2=0; \\
(xi)\,\,a_6 - a_5 k_1 - a_4 k_2=0;                                 & (xii)\,\,a_5 + a_6 k_1 + a_4 k_3=0;\\
(xiii)\,\,2 a_7 - 2 a_9 + a_8 k_1 + a_6 k_2 + a_5 k_3=0;           & (xiv)\,\, a_8 k_2 + a_7 k_3=0;\\
(xv)\,\, a_8 + a_9 k_1 + a_6 k_3=0;                                & (xvi)\,\,a_9 k_2+a_8 k_3=0;\\
(xvii)\,\,a_4 k_1=0;                                               & (xviii)\,\, a_7 k_2=0;\\
(xix)\,\,a_9 k_3=0;                                                & (xx)\,\, a_3 - a_5 - a_5 k_4 -  a_2 k_1 - a_1 k_2  +  \\
\quad                                                              & \quad\quad +2 a_7 r+ 2 a_4 s- a_5 s=0.
\end{array}
$$
It is straightforward to verify that $k_1=k_2=k_3=0.$ Thus, the equations $(xiv)$, $(xvi)$, $(xvii)$, $(xviii)$ and $(xix)$ simplify to $0 = 0$ and do not provide any additional information. Then using equations $(xii),$ $(xi),$ $(xv),$ $(ix)$ and $(xiii)$, we obtain that $a_5=a_6=a_8=a_2=0$ and $a_7=a_9,$ respectively. Since $a_2=0,$ equation $(i)$ implies that $a_0=0.$

Thus the system that we need to analyze reduces to the following relevant equations concerning the coefficients \( a_i \) and \( k_4 \):
\begin{multicols}{2}
\begin{itemize}
\item[$(ii)$] \( -a_1k_4 - a_1s = 0 \);
\item[$(iv)$] \( a_3 + a_3 k_4 = 0 \);
\item[$(vi)$] \( 2a_9+ a_9k_4 = 0 \);
\item[$(xx)$] \( a_3+2 a_9 r + 2 a_4 s = 0 \).
\item[$(iii)$] \( -2 a_9 c - a_1 s = 0 \);
\item[$(v)$] \( a_4k_4 + 2a_4s = 0 \);
\item[$(vii)$] \( 2a_9+ a_9k_4 = 0 \);
\end{itemize}
\end{multicols}
From equation $(vii)$ we conclude that $a_9=0$ or $k_4=-2.$

\noindent\textit{Case} 1: $a_9=0$. Since the invariant surface that we are looking for is of degree 2, we have that $a_4\neq0$. Thus, using equations $(iii),$ $(v)$, $(xx)$ and $(iv)$ we obtain that $a_1=0,$ $k_4=-2s,$ $a_3=-2sa_4$ and $s=1/2,$ respectively. Setting $a_4=1$, we get that $f(x,y,z)=x^2-z$ is the invariant surface $(a)$ given in Table \ref{tb_2}.

\noindent\textit{Case} 2: $k_4=-2$. From equation $(iv),$ we obtain that $a_3=0.$ In addition, equation $(v)$ implies that $a_4=0$ or $s=1.$

\noindent\textit{Subcase} 2.1: $a_4=0$. Since $a_9\neq 0$, equation $(xx)$ and $(iii)$ implies that $r=0$ and $sa_1=-2ca_9$, respectively. Furthermore, by equation $(ii)$, we get $a_1=0$ or $s=2.$
\begin{itemize}
\item If $s=2$, then taking $a_9=1$, we get that $f(x,y,z)=-cx+y^2+z^2$ is the invariant surface $(b)$ of Table \ref{tb_2}.
\item Otherwise, employing equation $(iii)$ we obtain that $a_1=0$ implies that $c=0$. Again, taking $a_9=1$, we obtain that $f(x,y,z)=y^2+z^2$ is the  invariant surface $(c)$ in Table \ref{tb_2}.
\end{itemize}

\noindent\textit{Subcase} 2.2: $s=1$. By equation $(ii),$ we know that $a_1=0.$ Thus, $(iii)$ implies that $c=0.$ In addition, from equation $(xx)$ we can conclude that $a_4=-a_9r.$ Setting $a_9=1$ we get that $f(x,y,z)=-rx^2+y^2+z^2$ is the invariant surface $(d)$ of Table \ref{tb_2}.
This completes the proof of Theorem \ref{prop_alg}.  
\end{proof}

Having proven Theorem $\ref{prop_alg}$ that outlines the invariant algebraic surfaces of degree 2 in the Ehrhard-Müller system, we now turn to Proposition \ref{p1}. This proposition demonstrates the existence of Darboux invariants for the system, offering further characterization of the system's dynamics

\begin{proof}[Proof of Proposition $\ref{p1}$]
We shall prove the statement (a) of this proposition, the other statements are proved in a similar way.

Let $I=I(x,y,z,t)= (x^2-z)e^t$. Then, when $s=1/2$ we have
$$
\dfrac{dI}{dt}=\dfrac{\p I}{\p x}\dfrac{dx}{dt}+\dfrac{\p I}{\p y}\dfrac{dy}{dt}+\dfrac{\p I}{\p z}\dfrac{dz}{dt}+\dfrac{\p I}{\p t} 
=xe^t(y-x) -e^t(xy-z)+ (x^2-z)e^t=0.
$$
Hence $(x^2-z)e^t$ is a Darboux invariant.
\end{proof}

 In what follows, we proceed to prove Proposition \ref{p2}. This one provides an important geometric insight by describing the phase portrait of the system on the infinity sphere, which reveals new structural features of the system.

\begin{proof}[Proof of Proposition $\ref{p2}$]
Using the results stated in Subsubsection \ref{s212} we obtain the expressions of the Ehrhard-M\"uller system \eqref{main_eq_1} in the local charts $U_j$ for $j=1,2,3$. Thus in the local chart $U_1$ we have the system
\begin{equation}\label{u1}
\begin{array}{l}
\dot z_1= -z_2 + r z_3 +(s-1)z_1 z_3+ c z_3^2 - s z_1^2 z_3, \vspace{0.2cm} \\
\dot z_2= z_1 +(s-1)z_2 z_3 - s z_1 z_2 z_3, \vspace{0.2cm} \\
\dot z_3=  s(1 - z_1)  z_3^2.
\end{array}
\end{equation}
The unique infinite equilibrium point of system \eqref{u1} is the $(0,0,0)$. Restricting system \eqref{u1} at infinity, i.e. at $z_3=0$, we obtain the system $\dot z_1=-z_2$, $\dot z_2=z_1$. Hence the equilibrium $(0,0,0)$ restricted to infinity is a linear center. So we have another linear center at the origin of the local chart $V_1$.
	
In the local chart $U_2$ we have the system
\begin{equation}\label{u2}
\begin{array}{l}
\dot z_1= s z_3 +(1-s)z_1 z_3 + z_1^2 z_2- r z_1^2 z_3 - c z_1 z_3^2, \vspace{0.2cm} \\
\dot z_2= z_1 + z_1 z_2^2 - r z_1 z_2 z_3 - c z_2 z_3^2, \vspace{0.2cm} \\
\dot z_3= z_3  (z_3+ z_1 z_2 - r z_1 z_3 - c z_3^2).
\end{array}
\end{equation}
System \eqref{u2} has the straight line $z_1=0$ filled with infinite equilibria.
	
Finally in the local chart $U_3$ we have the system
\begin{equation}\label{u3}
\begin{array}{l}
\dot z_1= (1-s)z_1 z_3 + s z_2 z_3 -z_1^2 z_2, \vspace{0.2cm} \\
\dot z_2= -z_1 + r z_1 z_3 - z_1 z_2^2 + c z_3^2, \vspace{0.2cm} \\
\dot z_3=  z_3 (z_3-z_1 z_2).
\end{array}
\end{equation}
In this local chart we only need to see if the origin of coordinates is an infinite equilibrium point, because all the other infinite equilibrium points already have been detected in the local charts $U_1$ and $U_2$. The origin $(0,0,0)$ of system \eqref{u3} is an infinite equilibrium point, that together with the equilibria of the local chart $U_2$, $V_2$ and $V_3$ form a circle filled with equilibria on the infinite sphere $\sss^2$. This completes the proof of Proposition \ref{p2}.
\end{proof}

\section{Proof of Theorem $\ref{ta}$}\label{sta}

Now we shall study the dynamics of the Ehrhard-M\"uller systems in the Poincaré ball having the invariant algebraic surface $x^2-z=0$.

First we analyze the dynamics of the differential system \eqref{main_eq_1} on the invariant surface $x^2-z=0$. Substituting $z=x^2$ in system \eqref{main_eq_1} we obtain the differential system
\begin{equation}\label{main_eq_c1_1_2}
\dot{x}=\dfrac{1}{2}(y-x),\quad
\dot{y}=rx-x^3-y+c.
\end{equation}
We begin by studying the local phase portraits of the infinite equilibrium points of system \eqref{main_eq_c1_1_2}.

\begin{figure}[h]
\begin{overpic}[scale=0.6]{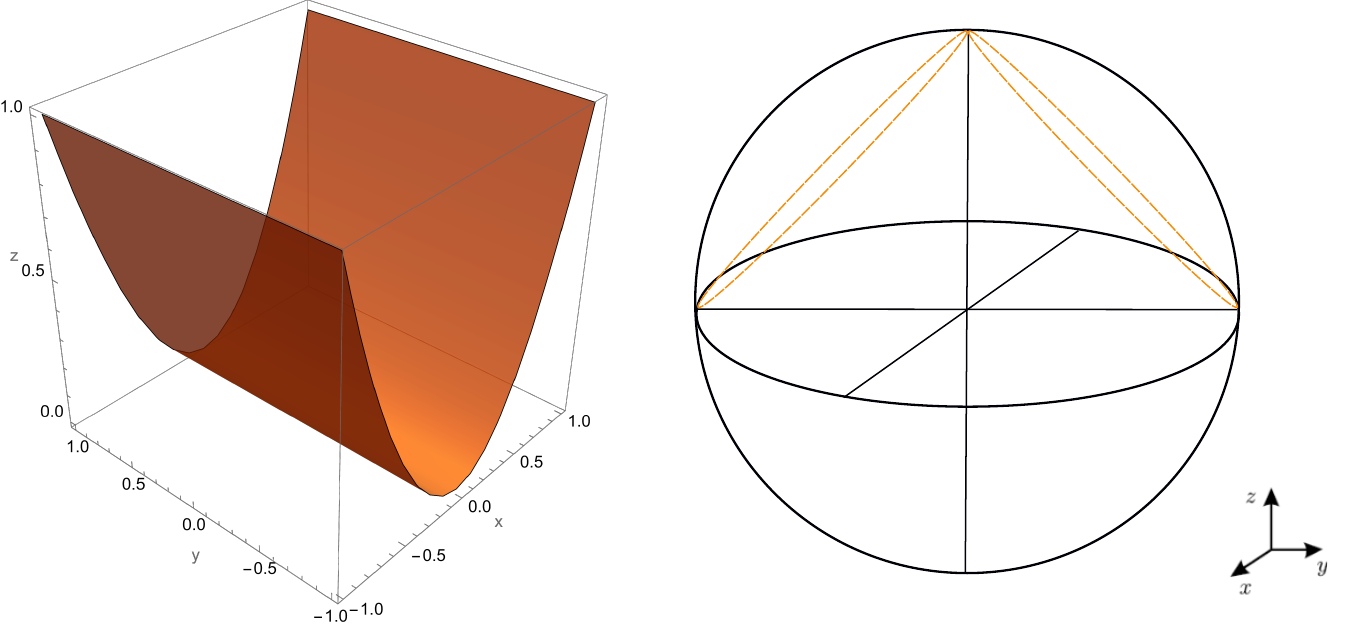}
             \put(25,-2){$(a)$}
        \put(73,-2){$(b)$}
		\end{overpic}
\caption{\footnotesize{(a) The invariant algebraic surface $x^2-z=0$. (b) The interception of the invariant algebraic surface $x^2-z=0$ with the sphere $\mathbb{S}^2$ of the infinity.}}
\label{F1}
\end{figure}

\begin{lemma}\label{L1}
The differential system \eqref{main_eq_c1_1_2} has two diametrically opposed infinite equilibrium points in the Poincaré disc at the origins of the local charts $U_2$ and $V_2$, that are formed by two hyperbolic sectors on the Poincaré sphere. Moreover, the intersection of the invariant surface $x^2-z=0$ with the sphere of the infinity $\mathbb{S}^2$ is shown in Figure \ref{F1}.
\end{lemma}

\begin{proof}
We study the infinite equilibrium points using the Poincaré compactification. From Subsubsection \ref{s211} system \eqref{main_eq_c1_1_2} in the local charts $U_1$ and $U_2$ becomes
\begin{equation}\label{main_eq_c1_1_U1_2}
\dot{u}=\dfrac{1}{2}(-2+2rv^2-uv^2-u^2v^2+2cv^3), \quad\dot{v}=\dfrac{1}{2}(1- u)v^3,
\end{equation}
and
\begin{equation}\label{main_eq_c1_1_U2_2}
\dot{u}=\dfrac{1}{2}(2u^4 + v^2 + uv^2-2ru^2v^2-2cuv^3),  \quad\dot{v}= -v(-u^3-v^2+ruv^2+cv^3),
\end{equation}
respectively. Since in the local chart $U_1$ system \eqref{main_eq_c1_1_2} has no infinite equilibrium points, it is enough to study if the origin of the local chart $U_2$ is an infinite equilibrium point, and it is a linearly zero equilibrium point. So in order to determine its local phase portrait we must use blow ups.
	
From \eqref{RR} we obtain for the differential system \eqref{main_eq_c1_1_U2_2} that $R_2(u, v) = v^3/2$. So $u = 0$ is not a characteristic direction. Hence we do the vertical blow-up $(u, v) = (u_1, u_1v_1)$ obtaining the system
\begin{equation}\label{main_eq_c1_1_U2_blow1_resc_trans_blow2_2c1}
\dot{u_1} =-\frac12 u_1^2 (-2 u_1^2 - v_1^2 - u_1 v_1^2 + 2 r u_1^2 v_1^2 + 2 c u_1^2 v_1^3), \quad \dot{v_1}=\dfrac12 u_1v^3_1(-1+u_1).
\end{equation}
Doing a rescaling of the time we eliminate the common factor $u_1$ between $\dot{u_1}$ and $\dot{v_1}$ and we obtain the system
\begin{equation}\label{main_eq_c1_1_U2_blow1_resc_trans_blow2_resc2_2c1}
\dot{u_1} =-\frac12 u_1 (-2 u_1^2 - v_1^2 - u_1 v_1^2 + 2 r u_1^2 v_1^2 + 2 c u_1^2 v_1^3), \quad \dot{v_1}=\dfrac12 v^3_1(-1+u_1).
\end{equation}
Then the unique equilibrium point of system \eqref{main_eq_c1_1_U2_blow1_resc_trans_blow2_resc2_2c1} on the straight line $u_1 = 0$ is the $(0, 0)$. We analyze its local phase portrait doing blow-ups. For system \eqref{main_eq_c1_1_U2_blow1_resc_trans_blow2_resc2_2c1} $R_2(u_1, v_1)= u_1v_1(v^2_1/2+u_1^2)$, so $u_1 = 0$ is a characteristic direction. Consequently before doing a vertical blow-up we translate the direction $u_1 = 0$ to the direction $v_1-u_1=0$ doing the change of variables $(u_1, v_1)=(v_2-u_2,v_2)$. In the new variables $(u_2,v_2)$ system \eqref{main_eq_c1_1_U2_blow1_resc_trans_blow2_resc2_2c1} becomes
\begin{equation}\label{main_eq_c1_1_U2_blow1_resc_rot_2c10}
\begin{array}{rl}
\dot{u_2}=& \dfrac12 \Big(2 u_2^3 - 6 u_2^2 v_2 + 7 u_2 v_2^2 - u_2^2 v_2^2 - 2 r u_2^3 v_2^2 - 4 v_2^3 + u_2 v_2^3 + 6 r u_2^2 v_2^3\\
& - 2 c u_2^3 v_2^3 - 6 r u_2 v_2^4 + 6 c u_2^2 v_2^4 + 2 r v_2^5 - 6 c u_2 v_2^5 + 2 c v_2^6\Big), \\
\dot{v_2}=& \dfrac12 v_2^3 (-1 - u_2 + v_2).
\end{array}
\end{equation}
Now we do the vertical blow-up $(u_2, v_2) = (u_3, u_3v_3)$ obtaining the system
\begin{equation}\label{main_eq_c1_1_u2_blow1_resc_trans_blow2_2c10}
\begin{array}{rl}
\dot{u_3}=& \dfrac12 u_3^3 \Big(2 - 6 v_3 + 7 v_3^2 - u_3 v_3^2 - 2 r u_3^2 v_3^2 - 4 v_3^3 + u_3 v_3^3 + 6 r u_3^2 v_3^3 - 2 c u_3^3 v_3^3\\
& - 6 r u_3^2 v_3^4 + 6 c u_3^3 v_3^4 + 2 r u_3^2 v_3^5 - 6 c u_3^3 v_3^5 + 2 c u_3^3 v_3^6\Big) ,\\
\dot{v_3}=&  -u_3^2 \Big(-1 + v_3) v_3 (-1 + 2 v_3 - 2 v_3^2 + r u_3^2 v_3^2 - 2 r u_3^2 v_3^3 + c u_3^3 v_3^3 + r u_3^2 v_3^4 \\
&- 2 c u_3^3 v_3^4 + c u_3^3 v_3^5\Big).
\end{array}
\end{equation}
Doing a rescaling of the time we eliminate the common factor $u^2_3$ between $\dot{u}_3$ and $\dot{v}_3$ and we obtain the system
\begin{equation}\label{main_eq_c1_1_u2_blow1_resc_trans_blow2_resc2_2c10}
\begin{array}{rl}
\dot{u_3}=& \dfrac12 u_3 \Big(2 - 6 v_3 + 7 v_3^2 - u_3 v_3^2 - 2 r u_3^2 v_3^2 - 4 v_3^3 + u_3 v_3^3 + 6 r u_3^2 v_3^3 - 2 c u_3^3 v_3^3\\
& - 6 r u_3^2 v_3^4 + 6 c u_3^3 v_3^4 + 2 r u_3^2 v_3^5 - 6 c u_3^3 v_3^5 + 2 c u_3^3 v_3^6\Big) ,\\
\dot{v_3}=&  -(-1 + v_3) v_3 \Big(-1 + 2 v_3 - 2 v_3^2 + r u_3^2 v_3^2 - 2 r u_3^2 v_3^3 + c u_3^3 v_3^3 + r u_3^2 v_3^4 \\
&- 2 c u_3^3 v_3^4 + c u_3^3 v_3^5\Big).
\end{array}
\end{equation}

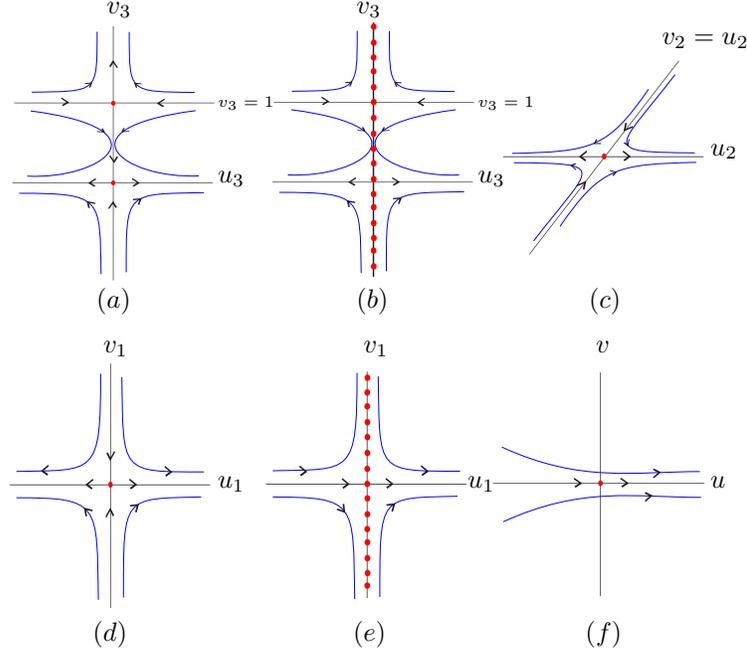
\begin{figure}[h]
\begin{overpic}[scale=0.5]{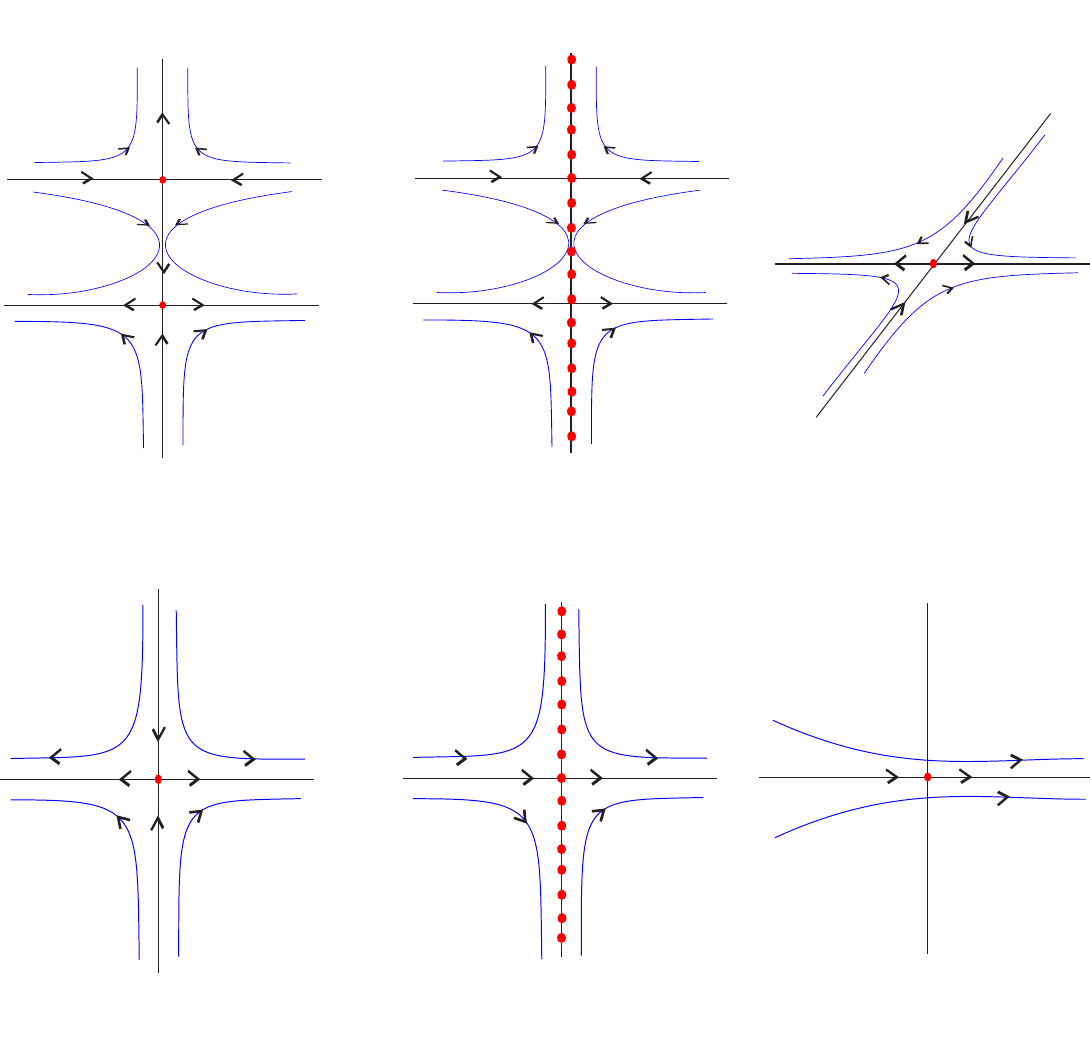}
        \put(101,24){$u$}
        \put(94,88){$v_2=u_2$}
        \put(66,24){$u_1$}
          \put(13.5,43.5){$v_1$}
        \put(67.5,68){$u_3$}
        \put(67.5,79){\tiny{$v_3=1$}}
        \put(50,92.5){$v_3$}
        \put(101,72){$u_2$}
        \put(84.5,43.5){$v$}
        \put(30,24){$u_1$}
        \put(51,43.5){$v_1$}
        \put(30,68){$u_3$}
        \put(30,79){\tiny{$v_3=1$}}
        \put(14,92.5){$v_3$}
          \put(50,50){$(b)$}
        \put(83.5,50){$(c)$}
        \put(12.5,50){$(a)$}
        \put(49.5,2){$(e)$}
        \put(83,2){$(f)$}
        \put(12,2){$(d)$}
		\end{overpic}
\caption{\footnotesize{The blow up of the origin of the local chart $U_2$.}}
\label{F4}
\end{figure}

The two equilibrium points of system
\eqref{main_eq_c1_1_u2_blow1_resc_trans_blow2_resc2_2c10} are the $(0, 0)$ and the $(0,1)$, both are hyperbolic saddles because the eigenvalues of the Jacobian matrix of the system at these equilibria are $-1$ and $1$, and $-1/2$ and $1$, respectively. See Figure \ref{F4}(a) for the phase portrait in a neighborhood of the straight line $u_3=0$. Hence the phase portrait in a neighborhood of the straight line $u_3=0$ for the differential system \eqref{main_eq_c1_1_u2_blow1_resc_trans_blow2_2c10} is shown in Figure \ref{F4}(b).
	
Going back through the blow up $(u_2, v_2) = (u_3, u_3v_3)$ and taking into account that $\dot u_2|_{u_2=v_2}=-v_2^3/2$, we obtain the local phase portrait at the equillibrium $(0,0)$ of the differential system \eqref{main_eq_c1_1_U2_blow1_resc_rot_2c10}, see Figure \ref{F4}(c). Undoing the twist $(u_1, v_1)=(u_2,v_2-u_2)$ we get the local phase portrait at the equillibrium $(0,0)$ of the differential system \eqref{main_eq_c1_1_U2_blow1_resc_trans_blow2_resc2_2c1}, see Figure \ref{F4}(d). Consequently, the local phase portrait at the equillibrium $(0,0)$ of the differential system \eqref{main_eq_c1_1_U2_blow1_resc_trans_blow2_2c1}, see Figure \ref{F4}(e). Finally, going back through the blow up $(u, v) = (u_1, u_1v_1)$ and taking into account that $\dot u|_{u=0}=v^2/2$, the local phase portrait at the origin of the local chart $U_2$ is shown in Figure \ref{F4}(f).
\end{proof}

\begin{figure}[h]
\begin{overpic}[scale=0.5]{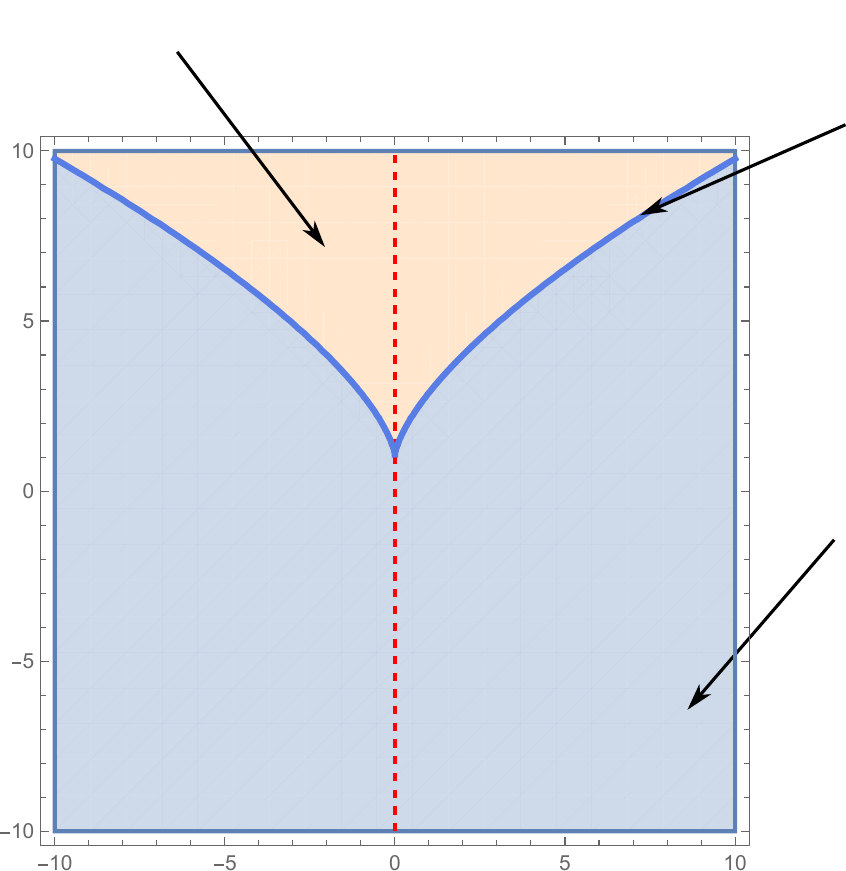}
\put(17,96){$R^+$}
\put(95,40){$R^-$}
\put(96,87){$\Delta\equiv0$}
\put(3,88){$r$}
\put(88,3){$c$}
\end{overpic}
\caption{Parameter plane $(c,r)$. The discriminant $\Delta\equiv0$ is the continuous curve given in the figure, that divides this plane into the regions $R^+=\{(c,r):\Delta(c,r)>0\}$ and $R^-=\{(c,r):\Delta(c,r)<0\}.$ The dashed line represents the straight line $c=0$, which was already studied in \cite{CZ,LMS}.}
\label{fig_box}
\end{figure}
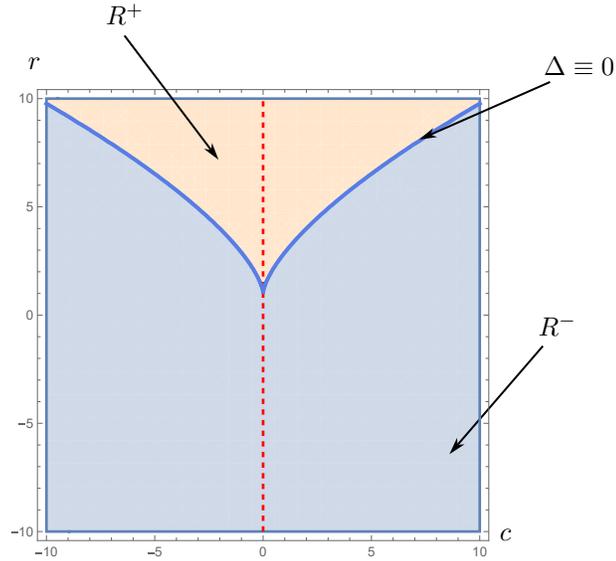

In what follows we will focus on studying the finite equilibrium points of system \eqref{main_eq_c1_1_2}. Recall that these equilibrium points  are on the diagonal $D=\{(x,y):x=y\}$ and satisfy the equation $$x^3+(1-r)x-c=0.$$ The discriminant of that polynomial is $\Delta(c,r)=-27c^2+4(r-1)^3,$ see Figure \ref{fig_box}. Using this discriminant we divide the study of the finite equilibria into three lemmas.

\begin{lemma}\label{L2}
Consider the region $R^+=\{(c,r):\Delta(c,r)>0\}$. Then system \eqref{main_eq_c1_1_2}$\Big|_{R^+}$ has three hyperbolic equilibrium points on the diagonal $D=\{(x,y):x=y\}$, one of them is a saddle and the other two are foci or nodes. More precisely,
\begin{itemize}
\item[(a)] if $1+8r-24x^2<0,$ then the equilibrium point $(x,x)$ is an attracting focus;
\item[(b)] if $1+8r-24x^2\in[0,9),$ then the equilibrium point $(x,x)$ is  an attracting node; and
\item[(c)] if $1+8r-24x^2>9,$ then the equilibrium point $(x,x)$ is a saddle.
\end{itemize}
See Figure \ref{F2} with $\Delta>0$.
\end{lemma}

\begin{proof}
Note that in the statement of this lemma $1+8r-24x^2$ can not be $9$, otherwise $\Delta(c,r)=0$.
	
Suppose that $p=(x,x)$ is an equilibrium point of \eqref{main_eq_c1_1_2}$\Big|_{R^+}$. Let $F$ be the vector field associated to system \eqref{main_eq_c1_1_2}.  The Jacobian matrix of system \eqref{main_eq_c1_1_2} at $p$ is
$$
JF(p)=\left(\begin{matrix}
-\dfrac{1}{2} & \dfrac{1}{2}\vspace{0.2cm}\\
r-3x^2 & -1
\end{matrix}\right).
$$
Moreover, their eigenvalues are given by $\lambda_{\pm}=\dfrac{1}{4}(-3\pm\sqrt{1+8r-24x^2})$ and since Tr$(JF(p))=-3/2\neq0,$ then system \eqref{main_eq_c1_1_2} has no center. So the statements of the lemma follows immediately from the Hartman-Grobman Theorem, see Theorem 2.15 of \cite{DLA} for the version of this theorem for the planar differential systems.
\end{proof}

\begin{lemma}\label{L3}
Consider the curve $\Delta=\{(c,r):\Delta(c,r)=0\}$. Then system \eqref{main_eq_c1_1_2} has one equilibrium point of multiplicity one and a second equilibrium point of multiplicity two on the diagonal $D=\{(x,y):x=y\}$, the first is a hyperbolic focus or a node, and the second is a semi-hyperbolic saddle-node. More precisely,
\begin{itemize}
\item[(a)] if $|c|>1/(8\sqrt{2})$, then the equilibrium point is an attracting focus; and
\item[(b)] if $|c|\le 1/(8\sqrt{2})$, then the equilibrium point is an attracting node.
\end{itemize}
See Figure \ref{F2} with $\Delta=0$.
\end{lemma}

\begin{proof}
Since  $\Delta=0, $ then $r=1+3(c/2)^{2/3}.$ Thus, the  equilibrium points of \eqref{main_eq_c1_1_2} are given by $p_1=\left(-(c/2)^{1/3},-(c/2)^{1/3}\right)$  and $p_2=\left((4c)^{1/3},(4c)^{1/3}\right)$. Let $F$ be the vector field associated to system \eqref{main_eq_c1_1_2}.  The Jacobian matrices of system \eqref{main_eq_c1_1_2} at $p_1$ and $p_2$ are
$$
JF(p_1)=\left(\begin{matrix}
-\dfrac{1}{2} & \dfrac{1}{2}\vspace{0.2cm}\\
1 & -1
\end{matrix}\right) \quad \text{and}\quad JF(p_2)=\left(\begin{matrix}
-\dfrac{1}{2} & \dfrac{1}{2}\vspace{0.2cm}\\
1-9\left(\dfrac{c}{2}\right)^{2/3} & -1
\end{matrix}\right),
$$
respectively. Moreover, their eigenvalues are $-3/2$ and $0$, and $(3/4)\Big(-1\pm\sqrt{1-(2^7c^2)^{1/3}}\Big)$, respectively. Since Tr$(JF(p_2))=-3/2\neq 0,$ then system \eqref{main_eq_c1_1_2} has no centers and by the Hartman-Grobman Theorem it follows statements (a) and (b). From \cite[Theorem 2.19]{DLA}, we can conclude that $p_1$ is a semi-hyperbolic saddle. See Figure \ref{F2}.
\end{proof}

\begin{lemma}\label{L4}
Consider the region $R^-=\{(c,r):\Delta(c,r)<0\}$. Then system \eqref{main_eq_c1_1_2}$\Big|_{R^-}$ has a unique equilibrium point with multiplicity one on the diagonal $D=\{(x,y):x=y\}$, which is a focus or node. Moreover,
\begin{itemize}
\item[(a)] If $1+8r-24(x^*)^2<0,$ then the equilibrium point $(x^*,x^*)$ is an attracting focus.
\item[(b)] If $1+8r-24(x^*)^2\in[0,9),$ then the equilibrium point $(x^*,x^*)$ is an attracting node.
\end{itemize}
Here
$$
x^*=-\frac{\sqrt[3]{2} (r-1)}{\sqrt[3]{\sqrt{729 c^2+4 (3-3 r)^3}-27 c}}-\frac{\sqrt[3]{\sqrt{729 c^2+4 (3-3 r)^3}-27
   c}}{3 \sqrt[3]{2}},
$$
with $729 c^2 + 4 (3 - 3 r)^3\ge 0$. See Figure \ref{F2} with $\Delta<0$.
\end{lemma}

\begin{proof}
It is easy to see that $p=(x^*,x^*)$ is the unique equilibrium point of system \eqref{main_eq_c1_1_2}$\Big|_{R^-}$. Let $F$ be the vector field associated to system \eqref{main_eq_c1_1_2}.  The Jacobian matrix of system \eqref{main_eq_c1_1_2} at $p$ is
$$
JF(p)=\left(\begin{matrix}
-\dfrac{1}{2} & \dfrac{1}{2} \vspace{0.2cm}\\
r-3(x^*)^2 & -1
\end{matrix}\right).
$$
Moreover, their eigenvalues are given by $\lambda_{\pm}=\dfrac{1}{4}(-3\pm\sqrt{1+8r-24(x^*)^2})$ and since Tr$(JF(p))=-3/2\neq0,$ then system \eqref{main_eq_c1_1_2} has no center.

Statements $(a)$ and $(b)$ follow immediately from the expression of the eigenvalues $\lambda_\pm$ and the Hartman-Grobman Theorem, see Figure \ref{F2}.

It is easy to show that when $\Delta(c,r)<0$ and  $729 c^2 + 4 (3 - 3 r)^3\ge 0$, it is not possible that $1+8r-24(x^*)^2\ge 9$.
\end{proof}

\begin{lemma}\label{L5}
System \eqref{main_eq_c1_1_2} has no limit cycles.
\end{lemma}

\begin{proof}
Let $F$ be the vector field associated to system \eqref{main_eq_c1_1_2}. Notice that the divergence of $F$ is $-3/2$, thus from \textit{Bendison's criterion} (see Section \ref{sec:ben_crit}) the result follows.
\end{proof}

From Lemmas \ref{L1}, \ref{L2}, \ref{L3}, \ref{L4} and $\ref{L5}$ follows statement (i) of Theorem \ref{ta}.

Let $\gamma(t)=(x(t),y(t),z(t))$ be an orbit of the differential system \eqref{main_eq_1} in the Poincaré ball outside the invariant surface $x^2-z=0$. Then the Darboux invariant $(x^2-z)e^t$ evaluated on this orbit is a non-zero constant. The interval of definition of an orbit of a differential system defined in a compact space is $(-\infty,\infty)$, see for instance the first chapter of \cite{DLA}. So the $\alpha$-limit and the $\omega$-limit of the orbit $\gamma(t)$ are defined. 

For the orbit $\gamma(t)$ and for all $t\in \R$ we have that $(x(t)^2-z(t))e^t=k$, where $k$ is a non-zero constant. Since $e^t\to 0$ when $t\to -\infty$ it follows that $x(t)^2-z(t) \to \pm \infty$ when $t\to -\infty$. Therefore, from Figure \ref{F1} we get that the $\alpha$-limit of the orbit $\gamma(t)$ contained in the interior of the Poincaré ball and in the region $x^2-z<0$ is the infinite equilibrium point at the origin of the local chart $U_3$, while the $\alpha$-limit of $\gamma(t)$ contained in the interior of the Poincaré ball and in the region $x^2-z>0$ is the infinite equilibrium point at the origin of the local chart $V_3$. This completes the proof of statement (ii) of Theorem \ref{ta}.

As before let $\gamma(t)=(x(t),y(t),z(t))$ be an orbit of the differential system \eqref{main_eq_1} in the Poincaré ball outside the invariant surface $x^2-z=0$. For this orbit and for all $t\in \R$ we have that $(x(t)^2-z(t))e^t=k$, where $k$ is a non-zero constant. Since $e^t\to \infty$ when $t\to \infty$ it follows that $x(t)^2-z(t) \to 0$ when $t\to \infty$. Hence, from Figures \ref{F2} and \ref{F1} we obtain that the $\omega$-limit of the orbit $\gamma(t)$ is either one of the stable finite equilibrium points contained in the surface $x^2-z=0$, or one of the two infinite equilibrium points at the endpoints of the $y$-axis contained in the surface $x^2-z=0$. So statement (iii) of Theorem \ref{ta} is proved.

\section{Proof of Theorem $\ref{tb}$}\label{stb}

We will analyze the dynamics of the Ehrhard-Müller systems having the invariant algebraic surface $y^2+z^2-cx=0$ in the Poincaré ball.

In this context we study the dynamics of the differential system \eqref{main_eq_1} restricted to the invariant algebraic surface $y^2+z^2-cx=0,$ which is given by
\begin{equation}\label{main_eq_c2}
\dot{y}=-\dfrac{z}{c}(y^2+z^2)-y+c,\quad
\dot{z}=\dfrac{y}{c}(y^2+z^2)-z.
\end{equation}
where $c\in\mathbb{R}\setminus\{0\}$. 

We begin by studying the nature of the points at the infinity of system \eqref{main_eq_c2}, for that we state the following result.
\begin{figure}[h]
\begin{overpic}[scale=0.5]{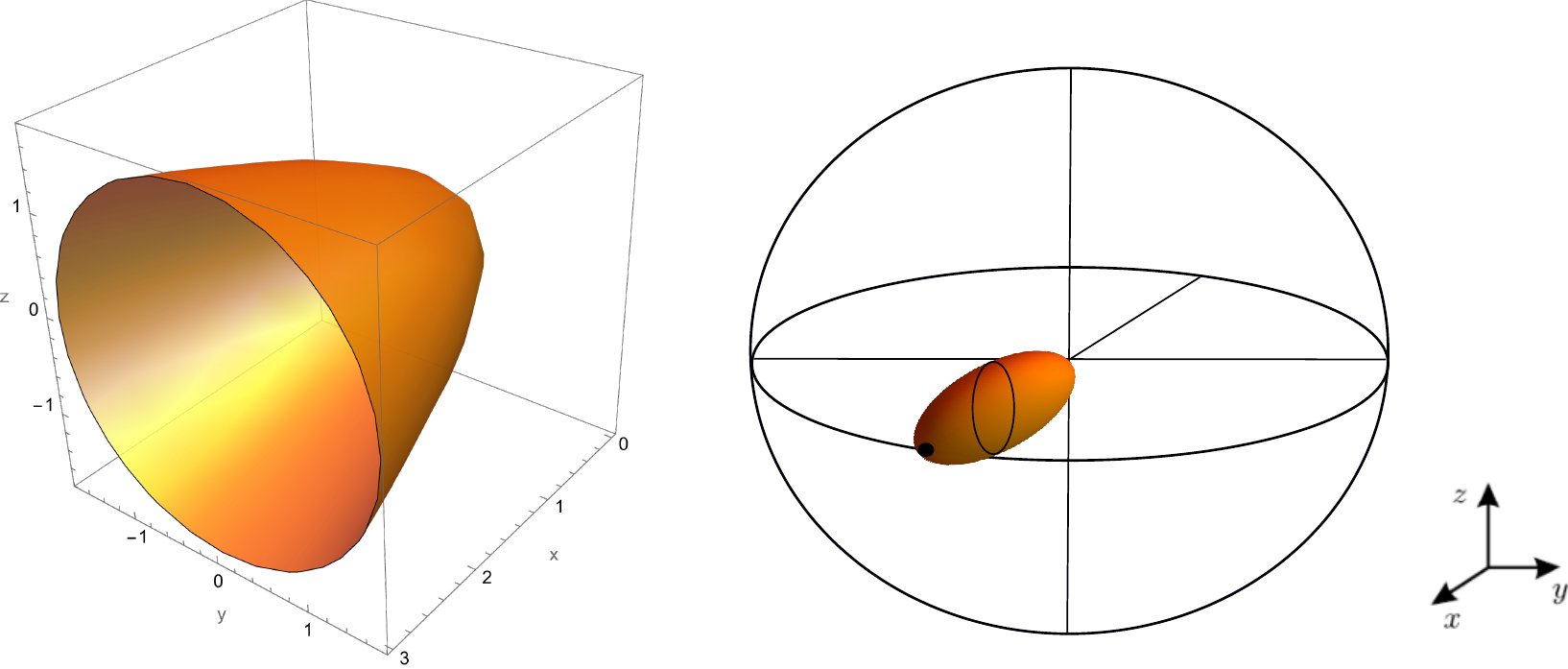}
        \put(57,11){$P$}
             \put(23,-2){$(a)$}
        \put(70,-2){$(b)$}
		\end{overpic}
\caption{\footnotesize{(a) The invariant algebraic surface $y^2+z^2-cx=0$. (b) The interception of the invariant algebraic surface $y^2+z^2-cx=0$ with the sphere $\mathbb{S}^2$ of the infinity. The equilibrium point $P$ at infinity is localized at the endpoint of the positive $x-$axis.}}
\label{F33}
\end{figure}
\begin{lemma}\label{compac_Rp_1}
The differential system \eqref{main_eq_c2} has no infinity equilibria. Furthermore, the intersection of the invariant surface $y^2+z^2-cx=0$ with the sphere of the infinity $\mathbb{S}^2$ is shown in Figure \ref{F33}. 
\end{lemma}
	
\begin{proof}
To obtain information about the infinite equilibria we use the Poincaré compactification. From Section \ref{poincare_comp} system \eqref{main_eq_c2} in the local charts $U_1$ and $U_2$ becomes
\begin{equation*}\label{main_eq_c2_U1}
\dot{u}=\dfrac{1}{c}((1+u^2)^2 - c^2 u v^3), \quad\dot{v}=-\dfrac{v}{c}(-u - u^3 - c v^2 + c^2 v^3),
\end{equation*}
and
\begin{equation*}\label{main_eq_c2_U2}
\dot{u}=\dfrac{1}{c}(-(1+u^2)^2+ c^2 v^3),  \quad\dot{v}= \dfrac{v}{c}(-u - u^3 + c v^2),
\end{equation*}
respectively. Since the equilibrium points at infinity (if they exist) are on $v=0,$ we obtain that $\dot{u}=\dfrac{1}{c}(1+u^2)^2\neq0$ (resp. $\dot{u}=-\dfrac{1}{c}(1+u^2)^2\neq0$) and $\dot{v}=0$ in the chart $U_1$ (resp. $U_2$). Consequently, in the local charts $U_1$ and $U_2$ system \eqref{main_eq_c2} has no infinite equilibrium points.
\end{proof}

Now, we analyze the finite equilibrium points of system~\eqref{main_eq_c2}. We shall study the equilibria lying in the set
$
B = \big\{(y, z) \in \mathbb{R}^2 \mid y^2 + z^2 - c y = 0,\ y \neq 0 \big\},$
which describes a circle centered at $(c/2, 0)$ with radius $c/2$ in the $yz$-plane, excluding the points $(0, 0)$ and $(c, 0)$.

\begin{lemma}\label{aux:l3}
System \eqref{main_eq_c1_1_2}$\Big|_{B}$ has a unique equilibrium point with multiplicity one, which is a stable focus.
\end{lemma}

\begin{proof}
Let $G$ be the vector field associated to system \eqref{main_eq_c2}.  The Jacobian matrix of system \eqref{main_eq_c2} at $p$ is
$$
JG(p)=\left(\begin{matrix}
-\dfrac{2yz}{c}-1 & -\dfrac{y^2+3z^2}{c}\\
\dfrac{z^2+3y^2}{c} & \dfrac{2yz}{c}-1
\end{matrix}\right).
$$
Moreover, their eigenvalues are given by $\lambda_{\pm}=-1\pm i\dfrac{\sqrt{3}}{c}(y^2+z^2)$ and since Tr$(JG(p))=-2\neq0,$ then system \eqref{main_eq_c2} has no center. 

The lemma directly follow from the expression of the eigenvalues \( \lambda_\pm \) and the Hartman-Grobman Theorem.		
\end{proof}
	
\begin{lemma}\label{PC_c}
The system described by equation \eqref{main_eq_c2} does not have periodic orbits.
\end{lemma}
	
\begin{proof}
Let \( F \) denote the vector field corresponding to the system in equation \eqref{main_eq_c2}. It is important to note that the divergence of \( F \) is equal to \(-2\). Therefore, by applying \textit{Bendixson's criterion} (see Section \ref{sec:ben_crit}), we conclude that the system has no limit cycles.
\end{proof}

From Lemmas \ref{compac_Rp_1}, \ref{aux:l3} and $\ref{PC_c}$ follows statement (i) of Theorem \ref{tb}.

Let \( \gamma(t) = (x(t), y(t), z(t)) \) be an orbit of the differential system \eqref{main_eq_1} in the Poincaré ball outside the invariant surface \( y^2 + z^2 - c x = 0 \). Then, the Darboux invariant \( (y^2 + z^2 - c x)e^{2t} \) evaluated on this orbit is a non-zero constant. The interval of definition of an orbit of a differential system defined in a compact space is \( (-\infty, \infty) \), see for instance the first chapter of \cite{DLA}. Therefore, the \( \alpha \)-limit and the \( \omega \)-limit of the orbit \( \gamma(t) \) are well-defined.

For the orbit \( \gamma(t) \) and for all \( t \in \mathbb{R} \), we have that \( (y(t)^2 + z(t)^2 - c x(t))e^{2t} = k \), where \( k \) is a non-zero constant. Since \( e^{2t} \to 0 \) as \( t \to -\infty \), it follows that \( y(t)^2 + z(t)^2 - c x(t) \to \pm \infty \) as \( t \to -\infty \). Hence, from Figure \ref{F33}, we obtain that the \( \alpha \)-limit of the orbit \( \gamma(t) \), contained in the interior of the Poincaré ball and in the region \( y^2 + z^2 - c x \neq 0 \), is the infinity equilibrium point \( P \) at the origin of the chart \( U_1 \), which is located at the endpoint of the positive \( x \)-axis. This completes the proof of statement (ii) of Theorem \ref{tb}.

As before, let \( \gamma(t) = (x(t), y(t), z(t)) \) be an orbit of the differential system \eqref{main_eq_1} in the Poincaré ball outside the invariant surface \( y^2 + z^2 - c x = 0 \). For this orbit and for all \( t \in \mathbb{R} \), we have that \( (y(t)^2 + z(t)^2 - c x(t))e^{2t} = k \), where \( k \) is a non-zero constant. Since \( e^{2t} \to \infty \) as \( t \to \infty \), it follows that \( y(t)^2 + z(t)^2 - c x(t) \to 0 \) as \( t \to \infty \). Hence, from Figures \ref{fig_surface_2} and \ref{F33}, we obtain that the \( \omega \)-limit of the orbit \( \gamma(t) \) is the stable focus $Q$ contained in the surface $y^2+z^2-cx=0$. Thus, statement (iii) of Theorem \ref{tb} is proved.

\section*{Acknowledgements}

Jaume Llibre is partially supported by the Agencia Estatal de Investigaci\'on of Spain grant PID2022-136613NB-100, AGAUR (Generalitat de Catalunya) grant 2021SGR00113, and by the Reial Acad\`emia de Ci\`encies i Arts de Barcelona. Gabriel Rondón is partially supported by  the Agencia Estatal de Investigaci\'on of Spain grant PID2022-136613NB-100 and by São Paulo Research Foundation (FAPESP) grants 2020-06708-9 and 2022-12123-9.

\section*{Author Declarations}

\textit{Conflict of Interest}: The authors declare that they have no known competing financial interests or personal relationships that could have appeared to influence the work reported in this paper.

\textit{Data Availability}: All data generated or analyzed during this study are included in this article.

\textit{Ethics Approval}: Not applicable.

\textit{Author Contributions}: 
\begin{itemize}
    \item Jaume Llibre: Conceptualization, Data curation, Formal analysis, Funding acquisition, Investigation, Methodology, Project administration, Resources, Software, Supervision, Validation, Visualization, Writing--original draft, Writing--review \& editing.
    \item Gabriel Rondón: Conceptualization, Data curation, Formal analysis, Funding acquisition, Investigation, Methodology, Project administration, Resources, Software, Supervision, Validation, Visualization, Writing--original draft, Writing--review \& editing.
\end{itemize}

\end{document}